\documentclass[11pt,a4paper]{article}
\usepackage{amssymb}
\usepackage{eurosym}
\usepackage{amsfonts}
\usepackage{amsmath}
\usepackage{amsthm}
\usepackage{graphicx}
\usepackage{float}
\usepackage{hyperref}

\hypersetup{
	colorlinks=true,
	linkcolor=blue,
	anchorcolor=blue,
	citecolor=blue
}
\newtheorem{theorem}{Theorem}[section]

\newtheorem{lemma}[theorem]{Lemma}

\theoremstyle{definition}
\newtheorem{definition}[theorem]{Definition}

\newtheorem{remark}[theorem]{Remark}

\newtheorem*{da}{Data availability}
\renewenvironment{proof}[1][Proof]{\noindent\textbf{#1.} }{\ \rule{0.5em}{0.5em}}

\renewcommand{\theequation}{\thesection.\arabic{equation}}
\allowdisplaybreaks
\input{tcilatex}

\def\limfunc#1{\mathop{\mathrm{#1}}}
\def\func#1{\mathop{\mathrm{#1}}\nolimits}

\def\dint{\displaystyle\int}

\def\Xint#1{\mathchoice
{\XXint\displaystyle\textstyle{#1}}%
{\XXint\textstyle\scriptstyle{#1}}%
{\XXint\scriptstyle\scriptscriptstyle{#1}}%
{\XXint\scriptscriptstyle\scriptscriptstyle{#1}}%
\!\int}
\def\XXint#1#2#3{{\setbox0=\hbox{$#1{#2#3}{\int}$ }
\vcenter{\hbox{$#2#3$ }}\kern-.6\wd0}}

\def\oint{\Xint-}

\def\enddoc{\end{document}}

\def\Qlb#1{#1}

\def\Qcb#1{#1} 

\def\FRAME#1#2#3#4#5#6#7#8
{
 \begin{figure}[H]
 \begin{center}
 \includegraphics[height=#3]{#7}
 \caption{#5}
 \label{#6}
 \end{center}
 \end{figure}
}

\begin{document}
	
	\title{Sharp long distance upper bounds for solutions of Leibenson’s equation on Riemannian manifolds}
	\author{Alexander Grigor'yan\and Jin Sun \and Philipp S\"urig}
	\date{March 2026}
	\maketitle
	
	\begin{abstract}
		We consider on Riemannian manifolds the Leibenson equation 
$\partial _{t}u=\Delta _{p}u^{q}$ that is also known as a doubly nonlinear evolution equation. We prove sharp upper estimates of weak subsolutions to this equation on Riemannian manifolds with non-negative Ricci curvature in the whole range of $p>1$ and $q>0$ satisfying $q(p-1)<1$. In this way, we improve the result of \cite{Grigoryan2024a} and prove Conjecture 1.2 from \cite{Grigoryan2024a}.
	\end{abstract}
	
	\let\thefootnote\relax\footnotetext{
		The first and the third author were funded by the Deutsche Forschungsgemeinschaft (DFG,
		German Research Foundation) - Project-ID 317210226 - SFB 1283.}
	
	\tableofcontents
	
	\section{Introduction}
	
	\setcounter{equation}{0}Let $M$ be an arbitrary Riemannian manifold. We consider solutions of the non-linear evolution
	equation 
	\begin{equation}
		\partial _{t}u=\Delta _{p}u^{q},  \label{evoeq}
	\end{equation}%
	where $p>1$, $q>0$, $u=u(x,t)$ is an unknown non-negative function of $x\in M$, $t\geq 0$, and 
	$\Delta _{p}$ is the Riemannian $p$-Laplacian $\Delta _{p}v=\func{div}\left( |\nabla v|^{p-2}\nabla v\right)$.
	
	The 
	equation (\ref{evoeq}) is frequently referred to as a \emph{doubly non-linear parabolic equation}. For the physical meaning of this equation  see \cite
	{grigor2024finite, leibenzon1945general, leibenson1945turbulent}. 
	
	When $M=\mathbb{R}^{n}$, G. I. Barenblatt \cite{barenblatt1952self}
	constructed for all $p>1, q>0$ spherically symmetric self-similar solutions of (\ref{evoeq}), that are nowadays called \textit{Barenblatt solutions}.
	
	Let us assume that \begin{equation}\label{orising}D:=1-q(p-1)>0.\end{equation}  If in addition \begin{equation*}
		\beta :=p-nD>0, \label{bcon}
	\end{equation*} then the Barenblatt solution  satisfies the estimate
	\begin{equation}\label{Barenintro}
		u(x, t)\simeq \dfrac{1}{%
			t^{n /\beta }} \left(1+ \dfrac{|x|}{t^{1 /\beta }}\right) ^{-\frac{p}{D}}
	\end{equation} (cf. Section 7.1 in \cite{Grigoryan2024a}), where the symbol "$\simeq $" means that the ratio of the terms is bounded from above and below by a positive constant.
	
	In \cite{Grigoryan2024a} two of the authors proved upper bounds for solutions of
	the Leibenson equation (\ref{evoeq}) on geodesically complete Riemannian manifolds in a subcase of (\ref{orising}). However, the long distance estimate obtained in this paper was not optimal.
	
	The purpose of the present paper is to obtain sharp estimates for solutions of (\ref{evoeq}) on Riemannian manifolds in the full range of $p$ and $q$ satisfying (\ref{orising}).	
	
	We understand solutions of  (\ref{evoeq}) in \(M\times \mathbb{R}_{+}\) in a certain weak sense (see Section \ref{secweaksing} for the definition).
	
 Denote by $\mu$ the \textit{Riemannian measure} on $M$, by $d$ the \textit{geodesic distance} and by $B(x, r)$ the \textit{geodesic ball} of radius $r$ centered at $x$. 
	
	The main result of the present paper is as follows (cf. \textbf{Theorem \ref{singmain}}).
	
	\begin{theorem}\label{singmainint}
		Let $M$ satisfy a relative
		Faber-Krahn inequality (see Section \ref{appFk} for definition) and assume that, for all $x\in M$ and all $R\geq 1,$%
		\begin{equation}\label{lowerball}
			\mu (B(x,R))\geq cR^{\alpha },
		\end{equation}%
		for some $c,\alpha >0.$ Assume that \emph{(\ref{orising})} holds and that 
		\begin{equation}\label{beta}
			\beta:=p-\alpha D >0.
		\end{equation}%
		Let $u$ be a bounded non-negative solution of\emph{(\ref{evoeq})} in $M\times [0, \infty)$ with initial function $u_{0}=u\left( \cdot ,0\right)\in L^{1}(M)\cap L^{\infty}(M)$. Set $A=\limfunc{supp}u_{0}$ and denote 
		$\left\vert x\right\vert =d(x,A).$
		Then, for all $t>0$ and all $x\in M$, we have%
		\begin{align}
			\left\Vert u\left( \cdot ,t\right) \right\Vert _{L^{\infty }(B(x,\frac{1}{2}\left\vert x\right\vert ))}\leq \frac{C}{t^{\alpha/\beta }} \left(1+\frac{|x|}{t^{1/\beta}}\right)^{-\frac{p}{D}},  \label{uoffonint}
		\end{align}%
		where the positive constants $C$ and $\gamma$ depend on $c, \alpha, p,q, ||u_{0}||_{L^{\infty}(M)}$, $||u_{0}||_{L^{1}(M)}$ and on the constants in the relative Faber-Krahn inequality.
	\end{theorem}
	
	In particular, if the solution   \(u\) is continuous then the left hand side of (\ref{uoffonint}) can be replaced by \(u(x,t)\). 
	
The relative Faber-Krahn inequality is satisfied if, for example, $M$ has non-negative Ricci curvature (see \cite{Buser, grigor, Saloff}).

Comparing the upper bound (\ref{uoffonint}) from Theorem \ref{singmainint} with the estimate (\ref{Barenintro}) of the Barenblatt solution, we see that the estimate (\ref{uoffonint}) is sharp in $\mathbb{R}^{n}$. A similar comparison takes place for some class of spherically symmetric manifolds (model manifolds) satisfying the relative Faber-Krahn inequality (see Remark \ref{remarksharp}).

In particular, our Theorem \ref{singmainint} improves the result of \cite{Grigoryan2024a} and implies Conjecture 1.2 from that paper. In this paper the following was proved. Let $M$ satisfy the relative Faber-Krahn inequality and (\ref{lowerball}). Assume that \begin{equation}\label{singularint} 
	1<p<2\quad\textnormal{and}\quad1\leq q< \frac{1}{p-1}
\end{equation}
(note that (\ref{singularint}) implies (\ref{orising})) and (\ref{beta}) hold. Then it was proved in \cite{Grigoryan2024a} that  \begin{equation}\label{oldsing}\left\Vert u\left( \cdot ,t\right) \right\Vert _{L^{\infty }(B(x,\frac{1}{2}\left\vert x\right\vert ))}\leq \frac{C}{t^{\alpha/\beta }} \Phi\left(1+\frac{|x|}{t^{1/\beta}}\right)\end{equation} where $$\Phi(s)=s^{-\frac{p}{D}}\log^{\gamma}(1+s),$$ where $\gamma$ is a positive constant. Hence, Theorem \ref{singmainint} improves the result of \cite{Grigoryan2024a} in two ways. We prove Theorem \ref{singmainint} in the whole range of $p$ and $q$ satisfying (\ref{orising}). In particular, in contrast to the result in \cite{Grigoryan2024a}, our Theorem \ref{singmainint} also holds when $p=2$, that is, when equation (\ref{evoeq}) becomes the \textit{porous medium equation} and (\ref{orising}) amounts to $q<1$. Secondly, we prove the estimate (\ref{oldsing}) with $\Phi(s)$ without the logarithmic term $\log^{\gamma}(1+s)$.

Let us discuss the differences in the methods of the proof of Theorem \ref{singmainint} and the result in \cite{Grigoryan2024a}.

The main technical lemma (Lemma \ref{TFPSqy}) in the present paper about the long distance decay of solutions of (\ref{evoeq}) says the following.
Let $u$ be a bounded non-negative solution of (\ref{evoeq}) in $M\times [ 0,\infty) $. Let $B=B\left( x_{0},R\right) $ be a ball such that the initial function $u(\cdot, 0)=u_{0}$ satisfies
\begin{equation*}
	u_{0}=0\ \text{in\ }B.
\end{equation*}%
Then, for all $t> 0$,
\begin{equation}\label{upperoffint}
	\left\Vert u\right\Vert _{L^{\infty }\left( \frac{1}{2}B\times [ 0,t] \right) }\leq C_{B}\left( \frac{t}{R^{p}}\right) ^{\frac{1}{D }},
\end{equation}%
where the positive constant $C_{B}$ depends on the intrinsic geometry of $B$ and $\gamma$ depends on $p, q$ and on the constants in the relative Faber-Krahn inequality.

In the proof of Lemma \ref{TFPSqy} we use a certain mean value inequality that is stated in Lemma \ref{specMV} that we borrowed from \cite{surig2026existence}.

In contrast to (\ref{upperoffint}), it was proved in \cite{Grigoryan2024a} under the above assumptions and under the additional restrictions (\ref{singularint}) that 
$$\left\Vert u\right\Vert _{L^{\infty }\left( \frac{1}{2}B\times [ 0,t] \right) }\leq C_{B}\left( \frac{t}{R^{p}}\right) ^{\frac{1}{D }}\log^{\gamma}\left(2+\left( \frac{R^{p}}{t}\right) ^{\frac{1}{D }}\frac{||u_{0}||_{L^{1}(M)}}{\mu \left(
	B\right) }\right).$$
The additional restrictions (\ref{singularint}) in \cite{Grigoryan2024a} came from the mean value inequality in \cite{Grigoryan2024} which was used in the proof.

Let us also discuss the differences in the method of the proof of the main lemmas about the long time decay of solutions of (\ref{evoeq}) and how the estimates hold for different ranges for $p$ and $q$. The main ingredient in both proofs is a non-linear mean inequality (see Lemma \ref{Tmeansing}) which says the following. 
Let $u$
be a non-negative bounded solution in 
$Q=B\times \left[ 0,T\right],$ $B=B(x_{0}, R)$, $T>0$.
Then, for the cylinder 
\begin{equation*}
	Q^{\prime }=\frac{1}{2}B\times [ \frac{1}{2}T,T] ,
\end{equation*}%
we have%
\begin{equation*}
	\left\Vert u\right\Vert _{L^{\infty }\left( Q^{\prime }\right) }\leq \left( 
	\frac{C_{B}S}{\mu (B)}\int_{Q}u^{\sigma}\right) ^{1/(\sigma+D)
	},  
\end{equation*}%
where $$S=\frac{\left\Vert u\right\Vert _{L^{\infty }(Q)}^{D }}{T}+%
\frac{1}{R^{p}},$$
$\sigma>0$ is any and the constant $C_{B}$ depends on \(p,q,\sigma\) and the intrinsic geometry of the ball $B$ (in fact, on the Faber-Krahn inequality in $B$).

Even though, in both papers, the estimate of the long time decay follows from this mean value inequality and a modification of the classical De Giorgi iteration argument \cite{de1957sulla}, the ranges of $p$ and $q$ for which the mean value inequality holds are different. 

The mean value inequality in \cite{Grigoryan2024} is proved in the case (\ref{singularint}). This is because the proof uses the fact that, if $u$ is a non-negative subsolution of (\ref{evoeq}), then the function 
\begin{equation}\label{utrun}
	(u^{a}-\theta)_{+}^{1/a}
\end{equation} 
is also a subsolution of (\ref{evoeq}), provided $\theta\geq 0$ and 
\begin{equation}\label{condforaint}a:=\frac{q(p-1)-1}{p-2}\in (0, 1].\end{equation} 
In particular, the condition \(a\in(0,1]\) in (\ref{condforaint}) is satisfied provided (\ref{singularint}) holds.

The second ingredient in this proof is a \textit{Caccioppoli-type inequality}, which says the following. Assume that $u$ is a subsolution of (\ref{evoeq}) in $B\times I$. Let $\eta \left( x,t\right) $ be a locally Lipschitz non-negative bounded function. Then for any $t_{1}, t_{2}\in I$ such that $t_{1}< t_{2}$ and any $\sigma$ large enough,
$$\left[ \int_{B }u^{\sigma+D }\eta ^{p}\right] _{t_{1}}^{t_{2}}+c_{1}%
\int_{B\times [t_{1}, t_{2}]}\left\vert \nabla \left( u^{\sigma/p }\eta \right) \right\vert
^{p}\leq \int_{B\times [t_{1}, t_{2}]}pu^{\sigma+D}\partial_{t}\eta \eta^{p-1}+c_{2}u^{\sigma }\left\vert \nabla \eta \right\vert ^{p}.$$ Then the aforementioned property allows to apply this inequality to the subsolutions $$u_{k}=\left( u^{a}-\left( 1-2^{-k}\right) \theta \right) _{+}^{1/a},\quad k\geq 0, $$ for some fixed $\theta>0$, where $a$ is given by (\ref{condforaint}).

However, in the present paper we prove the mean value inequality for the whole range $p>1, q>0$. This is because we use in the proof instead a Caccioppoli type inequality of the form $$\left[ \int_{B }u_{k+1}^{\lambda }\eta ^{p}\right] _{t_{1}}^{t_{2}}+A^{k}%
\int_{B\times [t_{1}, t_{2}]}\left\vert \nabla \left( u_{k+1}^{\sigma/p }\eta \right) \right\vert
^{p}\leq \int_{B\times [t_{1}, t_{2}]}pu_{k}^{\sigma+D}\partial_{t}\eta \eta^{p-1}+B^{k}u_{k}^{\sigma }\left\vert \nabla \eta \right\vert ^{p}.$$ where $$u_{k}=\left( u-\left( 1-2^{-k}\right) \theta \right) _{+},\quad k\geq 0,$$ and $A, B$ are positive constants (cf. Lemma \ref{Lem1}).

	The structure of the present paper is as follows.
	
	In Section \ref{secweaksing} we define the notion of a weak solution of the equation (\ref{evoeq}).
	
	In Section \ref{appFk} the aforementioned relative Faber-Krahn inequality is discussed.
	
	In Section \ref{seclongdis} we prove the main technical lemma (Lemma \ref{TFPSqy}) about the long distance decay of solutions of (\ref{evoeq}).
	
	In Section \ref{longtimedecsing} we prove the main lemma (Lemma \ref{Tlong}) about the long time decay of solutions of (\ref{evoeq}).
	
	For further qualitative properties for solutions (\ref{evoeq}) in the case (\ref{orising}) in various settings, we refer to \cite{bonforte2008fast, di1990non, dibenedetto2002current, surig2024finite, surig2025gradient}.
	
	In the case $D<0$ the Barenblatt solution has a finite propagation speed, and the same phenomenon occurs on arbitrary Riemannian manifolds (see \cite{bonforte2005asymptotics,  de2022wasserstein,  dekkers2005finite, Grigoryan2024, grigor2024finite, grillo2016smoothing, vazquez2015fundamental}). 
	
	In the borderline case $D=0,$ the Barenblatt solutions is positive but decays exponentially in distance. Similar  sub-Gaussian upper bounds of solutions of (\ref{evoeq}) on Riemannian manifolds were proved in the case \(D=0\) in \cite{surig2024sharp}. 
	
	We denote by $c, c^{\prime}, C, C^{\prime}$ positive constants whose value might change at each occurance.
	
				\begin{da} \normalfont
		This article has no associated data.
	\end{da}
	
	\section{Weak subsolutions}\label{secweaksing}
	
	We consider in what follows the following evolution
	equation on a Riemannian manifold $M$:%
	\begin{equation}
		\partial _{t}u=\Delta _{p}u^{q}.  \label{dtv}
	\end{equation}%
	By a \textit{subsolution} of (\ref{dtv}) we mean a non-negative function $u$
	satisfying 
	\begin{equation}
		\partial _{t}u\leq \Delta _{p}u^{q}  \label{subdtv}
	\end{equation}%
	in a certain weak sense as explained below.
	
	We assume throughout that 
	\begin{equation*}
		p>1\ \ \text{and}\ \ \ q>0.
	\end{equation*}%
	Set%
	\begin{equation*}
		D=1-q(p-1).
	\end{equation*}%
	
	Let $\mu $ denote the Riemannian measure on $M$. For simplicity of notation,
	we frequently omit in integrations the notation of measure. All integration
	in $M$ is done with respect to $d\mu $, and in $M\times \mathbb{R}$ -- with
	respect to $d\mu dt$, unless otherwise specified.
	
	Let $\Omega$ be an open subset of $M$ and $I$ be an interval in $[0, \infty)$.
	
	\begin{definition}
		\normalfont
		We say that a non-negative function $u=u(x, t)$ is a \textit{weak
			subsolution} of (\ref{dtv}) in $\Omega\times I$, if 
		\begin{equation}  \label{defvonsoluq}
			u\in C\left(I; L^{1+q}(\Omega)\right)\quad \textnormal{and}\quad 
			u^{q}\in L_{loc}^{p}\left(I; W^{1, p}(\Omega)\right)
		\end{equation}
		and (\ref{subdtv}) holds weakly in $\Omega\times I$, which means that for all $t_{1}, t_{2}\in I$ with $t_{1}<t_{2}$, and all non-negative \textit{test functions} 
		\begin{equation}  \label{defvontestsoluq}
			\psi\in W_{loc}^{1, 1+\frac{1}{q}}\left(I;
			L^{1+\frac{1}{q}}(\Omega)\right)\cap L_{loc}^{p}\left(I; W_{0}^{1,
				p}(\Omega)\right),
		\end{equation}
		we have 
		\begin{equation}  \label{defvonweaksolq}
			\left[\int_{\Omega}{u\psi}\right]_{t_{1}}^{t_{2}}+\int_{t_{1}}^{t_{2}}{%
				\int_{\Omega}{-u\partial_{t}\psi+|\nabla u^{q}|^{p-2}\langle\nabla u^{q},
					\nabla \psi\rangle}}\leq 0.
		\end{equation}
	\end{definition}
	
	Existence results for weak solutions of (\ref{dtv}) were obtained in \cite{andreucci1990new, benilan1995strong, coulhon2016regularisation, ivanov1997regularity,  ishige1996existence} in the euclidean setting and in \cite{andreucci2015optimal, de2022wasserstein, surig2026existence} on manifolds.
	
	If $u$ is of the class (\ref{defvonsoluq}) then \(\nabla(u^{q})\) is defined as an element of \(L^{p}(\Omega)\). Then we define \(\nabla u\ \)as follows:\textbf{} 
	\begin{equation*}
		\nabla u:=\left\{ 
		\begin{array}{ll}
			q^{-1}u^{1-q}\nabla(u^{q}), & u>0, \\ 
			0, & u=0.%
		\end{array}%
		\right.
	\end{equation*}

\begin{lemma}\cite{surig2026existence}
	\label{Lem1} Let $u=u\left( x,t\right) $ be a
	non-negative bounded subsolution to \emph{(\ref{dtv})} in a cylinder $\Omega\times (0, T)$.
	Let $\eta \left( x,t\right) $ be a locally Lipschitz non-negative bounded
	function in $\Omega\times [0, T]$ such that $\eta \left( \cdot ,t\right) $ has
	compact support in $\Omega $ for all $t\in [0, T]$. Fix some real $\sigma$
	such that 
	\begin{equation}
		\sigma \geq \max(p, pq)  \label{la>3-m}
	\end{equation}%
	and set 
	\begin{equation}
		\lambda =\sigma +D \ \ \ \ \text{and}\ \ \ \ \ \alpha =\dfrac{\sigma }{p%
		}.  \label{alpha}
	\end{equation}%
	Choose $0\leq t_{1}<t_{2}\leq T$ and set $Q=\Omega \times \left[ t_{1},t_{2}%
	\right] $. Then, for any $\theta_{1}>\theta_{0}>0$,
	\begin{align}
		\left[ \int_{\Omega }(u-\theta_{1})_{+}^{\lambda }\eta ^{p}\right] _{t_{1}}^{t_{2}}&+c_{1}\left(\frac{\theta_{1}}{\theta_{1}-\theta_{0}}\right)^{-(q-1)(p-1)_{-}}
		\int_{Q}\left\vert \nabla \left( (u-\theta_{1})_{+}^{\alpha }\eta \right) \right\vert
		^{p}\\&\leq \int_{Q}\left[ p(u-\theta_{0})_{+}^{\lambda }\eta ^{p-1}\partial _{t}\eta
		+c_{2}\left(\frac{\theta_{1}}{\theta_{1}-\theta_{0}}\right)^{(q-1)(p-1)_{+}}(u-\theta_{0})_{+}^{\sigma }\left\vert \nabla \eta \right\vert ^{p}\right] ,
		\label{veta1}
	\end{align}%
	where $c_{1},c_{2}$ are positive constants depending on $p$, $q$, $\lambda $.
\end{lemma}

Let us recall for later that \begin{equation}\label{valpha}v^{\alpha}=(u-\theta_{1})_{+}^{\alpha}\in L^{p}\left((0, T); W_{0}^{1, p}(\Omega)\right).\end{equation}
Indeed, using $\alpha\geq q$, we get that the function $\Phi(s)=s^{\frac{\alpha}{q}}$ is Lipschitz on any bounded interval in $[0, \infty)$. Thus, $u^{\alpha}=\Phi(u^{q})\in W_{0}^{1, p}(\Omega)$ and $\left|\nabla u^{\alpha}\right|=\left|\Phi^{\prime}(u^{q})\nabla u^{q}\right|\leq C\left|\nabla u^{q}\right|$, whence \begin{equation}\label{valpha0}\int_{Q}u^{\alpha p}+\left\vert \nabla  u^{\alpha } \right\vert^{p}\leq \left(||u||_{L^{\infty}(Q)}^{\sigma-pq}+C\right) \int_{Q}u^{pq}+\left\vert \nabla  u^{q } \right\vert^{p}.\end{equation}
Therefore, $|\nabla v^{\alpha}|=\alpha v^{\alpha-1}|\nabla v|\leq \alpha u^{\alpha-1}|\nabla u|=|\nabla u^{\alpha}|$, since $\alpha\geq 1$ by (\ref{la>3-m}) and $v\leq u$.
	
	\begin{lemma}
		\label{monl1}\emph{\cite{Grigoryan2024}}
		Let $M$ be geodesically complete and $v=v\left( x,t\right) $ be a bounded non-negative
		subsolution to \emph{(\ref{dtv})} in $M\times I$. For any $\lambda\in[1,\infty]$,  the function 
		\begin{equation*}
			t\mapsto \left\Vert v(\cdot ,t)\right\Vert _{L^{\lambda}(M)}
		\end{equation*}%
		is monotone decreasing in $I$.
	\end{lemma}
	
	\section{Faber-Krahn inequality}\label{appFk}
	Let $M$ be a connected Riemannian manifold of dimension $n$ and $d$ be the geodesic distance on $M$. For any $%
	x\in M$ and $r>0$, denote by $B(x,r)$ the geodesic ball of radius $r$
	centered at $x$, that is,%
	\begin{equation*}
		B(x,r)=\left\{ y\in M:d(x,y)<r\right\} .
	\end{equation*}
	
	Let the ball $B$ be precompact.
	Then the following \textit{Faber-Krahn inequality} in $B$
	of order $p\geq 1$ holds: if $w\in {W}_{0}^{1,p}(B)$ is non-negative,
	$$E=\left\{ w>0\right\}$$ and $r(B)$ denotes the radius of the ball $B$,
	then%
	\begin{equation}
		\dint_{B}\left\vert \nabla w\right\vert ^{p}\geq \dfrac{1}{r(B)^{p}}\left( \iota (B)\dfrac{\mu (B)}{\mu (E)}\right) ^{\nu
		}\dint_{B}w^{p},   \label{FKp}
	\end{equation}%
	where $\nu>0$ and $\iota(B)$ is a positive constant that depends on the geometry of $B$. The value of $\nu$ is independent of $B$ and can be chosen as follows:
	\begin{equation}\label{nu}
		\nu =\left\{ 
		\begin{array}{ll}
			\dfrac{p}{n}, & \text{if }n>p, \\ 
			\text{any number}\in (0, 1), & \text{if }n\leq p.%
		\end{array}%
		\right.
	\end{equation} 
	
	Choosing $\iota (B)$ to be an optimal constant in (\ref{FKp}) we obtain that the function%
	\begin{equation}
		B\mapsto \frac{\left( \iota (B)\mu (B)\right) ^{\nu }}{r(B)^{p}}
		\label{cmu}
	\end{equation}%
	is monotone decreasing with respect to the partial order $\subset $ on
	balls.
	
	We say that $M$ satisfies a \textit{relative Faber-Krahn inequality} of order $p$ if (\ref{FKp}) holds with $\iota(B)\geq \textnormal{const}>0$ for all geodesic balls $B$. For example, this holds if $M$ is complete, non-compact and satisfies $Ricci_{M}\geq 0$ (see \cite{Buser, grigor, Saloff}).
	
	\section{Long distance decay}\label{seclongdis}
	
			From now on we always assume that
	\begin{equation}\label{signconpq}
		D=1-q(p-1)>0.
	\end{equation}
	
	\begin{lemma}\label{specMV}
		Let the ball $B=B\left( x_{0},R\right) $ be precompact. Let $u$
		be a non-negative bounded subsolution in 
	$Q=B\times \left[ 0,t\right]$
		such that 
		\begin{equation*}
			u\left( \cdot ,0\right) =0\ \text{in }B.
		\end{equation*}%
		Let $\sigma$ and $\lambda$ be reals such that \begin{equation}\label{silapos}\sigma>0\quad \textnormal{and}\quad \lambda=\sigma+D.\end{equation}
		Then, for the cylinder 
	$Q^{\prime }=\frac{1}{2}B\times \left[ 0,t\right] ,$
		we have%
		\begin{equation}
			\left\Vert u\right\Vert _{L^{\infty }\left( Q^{\prime }\right) }\leq \left( 
			\frac{C}{\iota (B)\mu (B)R^{p}}\int_{Q}u^{\sigma}\right) ^{1/\lambda
			},  \label{mean}
		\end{equation}%
		where $\iota (B)$ is the Faber-Krahn constant in $B$,
		and the constant $C$ depends on $p$, $q$, $\lambda $ and the Faber-Krahn exponent $\nu$.
	\end{lemma}
	
	\FRAME{ftbpF}{2.1392in}{1.6447in}{0pt}{\Qcb{Cylinders $Q$ and $Q^{\prime }$}}{\Qlb{pic3}}{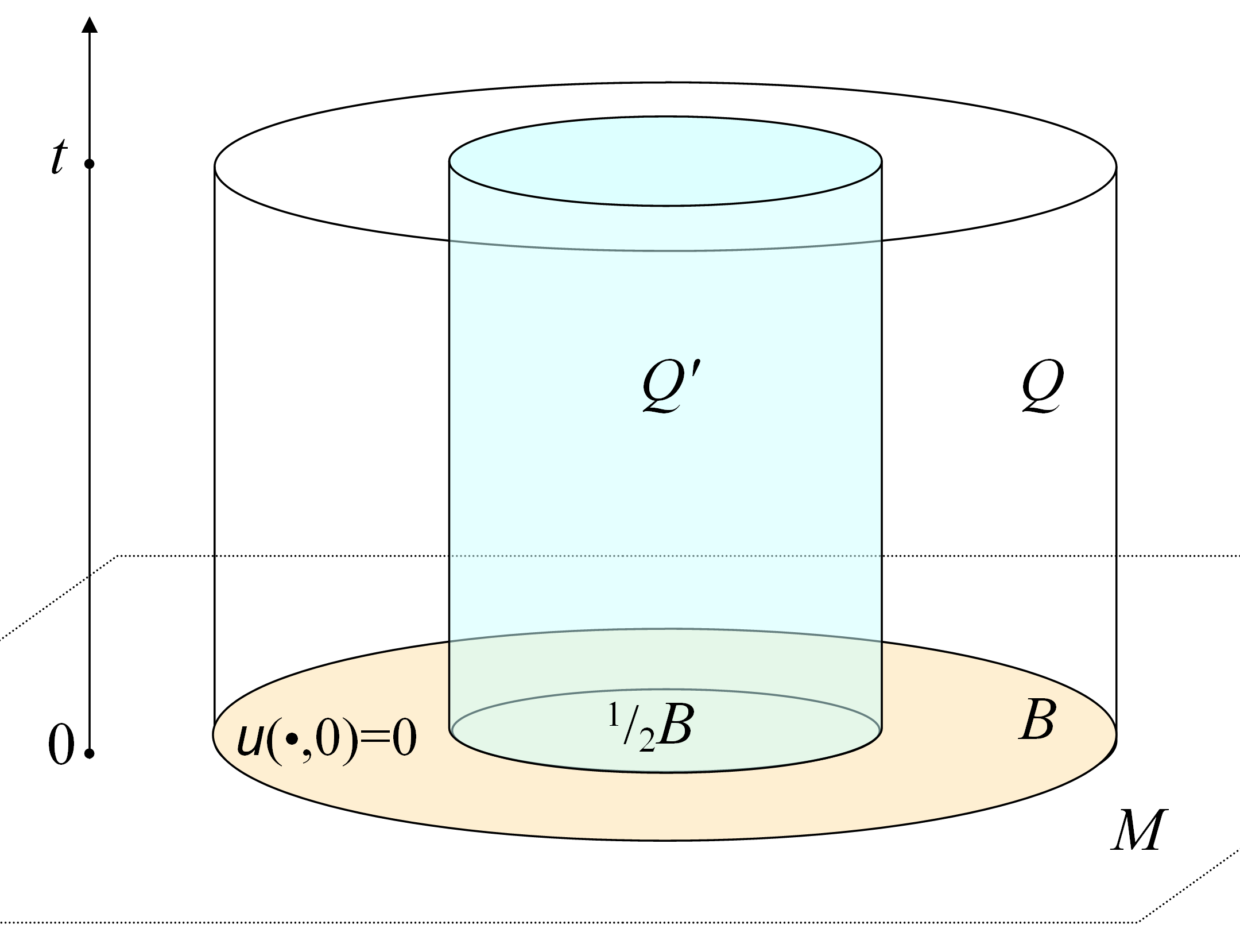}{\special%
		{language "Scientific Word";type "GRAPHIC";maintain-aspect-ratio
			TRUE;display "USEDEF";valid_file "F";width 2.1392in;height 1.6447in;depth
			0pt;original-width 10.0838in;original-height 7.7406in;cropleft "0";croptop
			"1";cropright "1";cropbottom "0";filename 'pic3.png';file-properties
			"XNPEU";}}
		
	\begin{remark}
	In \cite{Grigoryan2024a} Lemma \ref{specMV} was taken from \cite{Grigoryan2024} where it was proved under the additional condition $$p<2\quad\textnormal{and}\quad q\geq 1.$$
	\end{remark}	

\begin{proof}
This Lemma is already proved in \cite{surig2026existence}. However, since the convergence of the constant $C$ in (\ref{mean}) for $\sigma\to 0+$ is not adressed there and we need this for later usage, we give a shortened proof here.

It was shown in the first part of the proof of Lemma 4.5 in \cite{surig2026existence} that for any large enough $\sigma$, say $\sigma\geq\sigma_{0}$, with $\sigma_{0}$ only depending on $p$ and $q$, that \begin{equation}\label{firstparttaken}\left\Vert u\right\Vert _{L^{\infty }\left( Q^{\prime }\right) }\leq \left( 
\frac{C}{\iota (B)\mu (B)R^{p}}\int_{Q}u^{\sigma}\right) ^{1/\lambda
},\end{equation} where $\lambda=\sigma+D$ and $C$ depends on $p, q, \nu $ and $\sigma$. 

Now we prove (\ref{firstparttaken}) for any $\sigma>0$ (cf. \cite{Grigoryan2024}). Let $\sigma <\sigma_{0}$ and denote 
\begin{equation*}
	\lambda_{0}=\sigma_{0}+D\ \ \ \text{and\ \ }\lambda=\sigma+D
\end{equation*}%
so that $\lambda<\lambda_{0}.$

For simplicity of notation, for any set $E\subset M$, denote
$E^{t}=E\times \left[ 0,t\right]$.

In particular, (\ref{firstparttaken}) implies that, for any precompact ball $B$ of
radius $R$, 
\begin{equation}\label{sigma0}
	\left\Vert u\right\Vert _{L^{\infty }(\frac{1}{2}B^{t})}^{\lambda _{0}}\leq 
	\frac{C(\sigma_{0})}{\chi (B)R^{p}}\int_{B^{t}}u^{\sigma _{0}},
\end{equation}%
where $C(\sigma_{0})$ depends on $p, q, \nu$ and $\sigma_{0}$, 
$\chi (B):=\iota (B)\mu (B)$ and according to our notation $B^{t}=B\times [0, t]$.
Consider for $k\geq 0$, a sequence 
$$R_{k}=\left(1-\frac{1}{2^{k+1}}\right)R,$$
so that $R_{0}=\frac{1}{2}R$ and $R_{k}\uparrow R$ as $k\rightarrow \infty $%
, and set 
$B_{k}=B(x_{0},R_{k})$.
Denoting also $B=B(x_{0},R),$ we see that%
\begin{equation*}
	\frac{1}{2}B\subset B_{k}\subset B\ \ \text{and\ \ }B_{k}\uparrow B
\end{equation*}
as $k\rightarrow \infty .$ Set also 
$\rho _{k}=R_{k+1}-R_{k}=\frac{1}{2^{k+2}}R$.
\FRAME{ftbphF}{2.1662in}{1.8558in}{0pt}{\Qcb{Balls $B_{k}$ and $B(x, \rho_{k})$}}{\Qlb{pic8}}{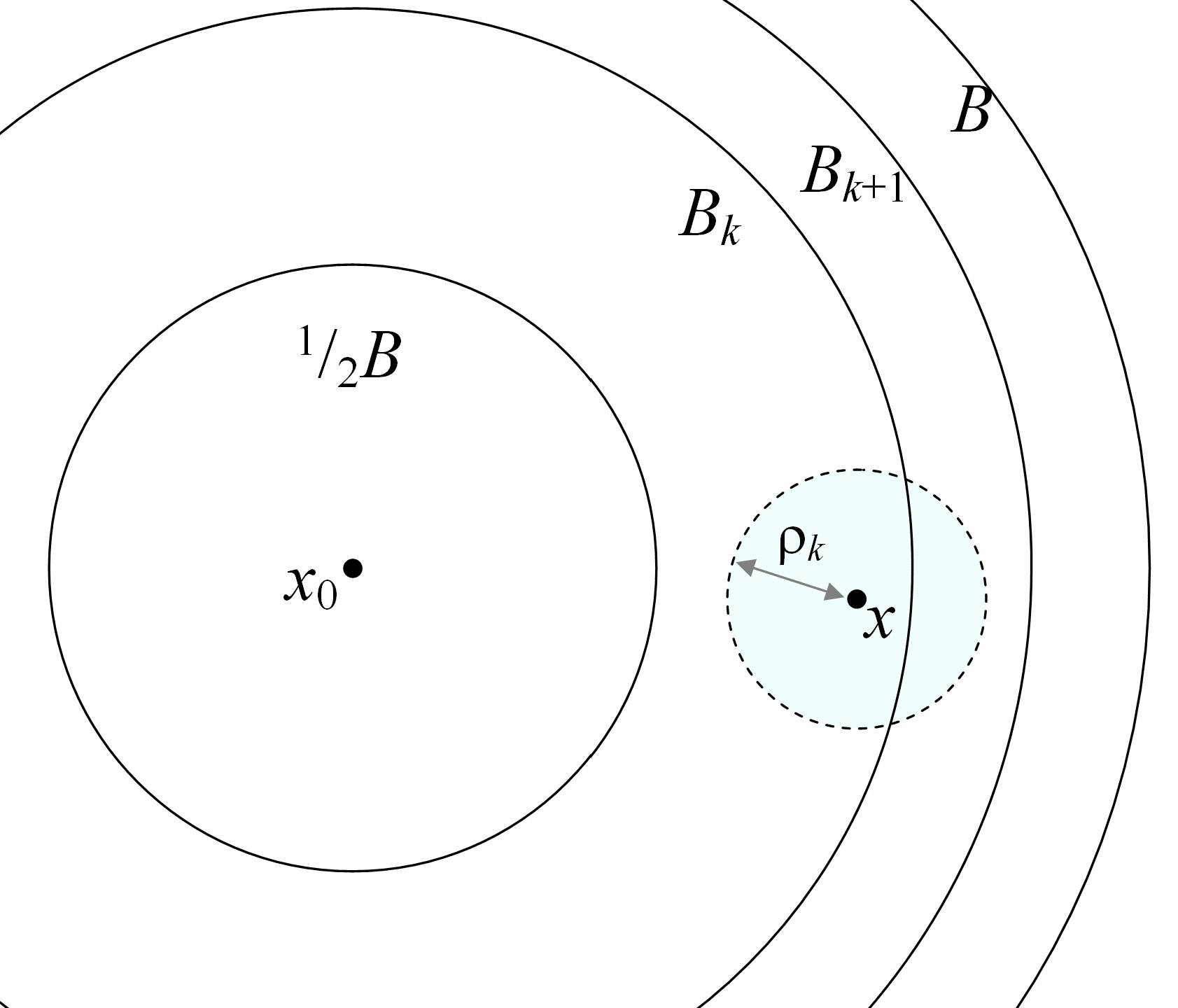}{\special%
	{language "Scientific Word";type "GRAPHIC";maintain-aspect-ratio
		TRUE;display "USEDEF";valid_file "F";width 2.1662in;height 1.8558in;depth
		0pt;original-width 7.8574in;original-height 6.7229in;cropleft "0";croptop
		"1";cropright "1";cropbottom "0";filename 'pic8.png';file-properties
		"XNPEU";}}

For any point $x\in B_{k}$, applying (\ref{sigma0}) in the ball $%
B\left( x,\rho _{k}\right) $, we obtain%
\begin{align*}
	\left\Vert u\right\Vert _{L^{\infty }(B^{t}(x,\frac{1}{2}\rho
		_{k}))}^{\lambda _{0}}& \leq \frac{C(\sigma_{0})}{\chi \left( B(x,\rho _{k})\right) \rho
		_{k}^{p}}\int_{B^{t}(x,\rho _{k})}u^{\sigma _{0}} \\
	& \leq \frac{C(\sigma_{0})}{\chi \left( B(x,\rho _{k})\right) \rho _{k}^{p}}\left\Vert
	u\right\Vert _{L^{\infty }(B^{t}(x,\rho _{k}))}^{\sigma _{0}-\sigma
	}\int_{B^{t}(x,\rho _{k})}u^{\sigma }.
\end{align*}%
Since 
$B\left( x,\rho _{k}\right) \subset B_{k+1}\subset B$,
we have by the monotonicity of (\ref{cmu}) 
\begin{equation*}
	\frac{\chi (B(x,\rho _{k}))}{\rho _{k}^{p/\nu }}\geq \frac{\chi \left(
		B\right) }{R^{p/\nu }}
\end{equation*}%
whence%
\begin{equation*}
	\frac{1}{\chi (B(x,\rho _{k}))}\leq \frac{\left( R/\rho _{k}\right)
		^{p/\nu }}{\chi (B)}=\frac{2^{\left( k+2\right) p/\nu }}{\chi (B)}.
\end{equation*}%
Hence, we obtain 
\begin{equation*}
	\left\Vert u\right\Vert _{L^{\infty }(B^{t}(x,\frac{1}{2}\rho
		_{k}))}^{\lambda _{0}}\leq \frac{C(\sigma_{0})2^{kp(\nu ^{-1}+1)}}{\chi \left( B\right)
		r^{p}}\left\Vert u\right\Vert _{L^{\infty }(B_{k+1}^{t})}^{\lambda
		_{0}-\lambda }\int_{B^{t}}u^{\sigma }.
\end{equation*}%
Covering $B_{k}$ by a sequence of balls $B(x,\frac{1}{2}\rho _{k})$ with $%
x\in B_{k}$, we obtain%
\begin{equation}
	\left\Vert u\right\Vert _{L^{\infty }(B_{k}^{t})}^{\lambda _{0}}\leq \frac{%
		C(\sigma_{0})2^{kp(\nu ^{-1}+1)}}{\chi \left( B\right) R^{p}}\left\Vert u\right\Vert
	_{L^{\infty }(B_{k+1}^{t})}^{\lambda _{0}-\lambda }\int_{B^{t}}u^{\sigma }.
	\label{Bkt}
\end{equation}%
Setting 
$J_{k}=\left\Vert u\right\Vert _{L^{\infty }(B_{k}^{t})}^{-\left( \lambda
	_{0}-\lambda \right) },$        
we rewrite (\ref{Bkt}) as follows:%
\begin{equation*}
	J_{k+1}\leq \frac{A^{k}}{\Theta }J_{k}^{\frac{\lambda _{0}}{\lambda
			_{0}-\lambda }}=\frac{A^{k}}{\Theta }J_{k}^{1+\omega },
\end{equation*}%
where
\begin{equation*}
A =2^{p(\nu ^{-1}+1)}, \quad \Theta ^{-1} =\frac{C(\sigma_{0})}{\chi \left( B\right) R^{p}}\int_{B^{t}}u^{\sigma }\quad \textnormal{and}\quad\omega =\frac{\lambda _{0}}{\lambda _{0}-\lambda }-1=\frac{\lambda }{%
	\lambda _{0}-\lambda }.	
\end{equation*} 
Applying Lemma 5.2 from \cite{Grigoryan2024}, we obtain%
\begin{equation*}
	J_{k}\leq \left( \frac{J_{0}}{\left( A^{-1/\omega }\Theta \right) ^{1/\omega
	}}\right) ^{\left( 1+\omega \right) ^{k}}\left( A^{-1/\omega }\Theta \right)
	^{1/\omega },
\end{equation*}%
that is, 
\begin{equation*}
	J_{0}\geq \left( A^{-1/\omega }\Theta \right) ^{1/\omega }\left( \left(
	A^{1/\omega }\Theta ^{-1}\right) ^{1/\omega }J_{k}\right) ^{\frac{1}{\left(
			1+\omega \right) ^{k}}}.
\end{equation*}%
Since 
$J_{k}\geq \left\Vert u\right\Vert _{L^{\infty }(B^{t})}^{-\left( \lambda
	_{0}-\lambda \right) }>\func{const}>0$,
we see that%
\begin{equation*}
	\liminf_{k\rightarrow \infty }\left( \left( A^{1/\omega }\Theta ^{-1}\right)
	^{1/\omega }J_{k}\right) ^{\frac{1}{\left( 1+\omega \right) ^{k}}}\geq 1,
\end{equation*}%
whence%
\begin{equation*}
	J_{0}\geq \left( A^{-1/\omega }\Theta \right) ^{1/\omega }.
\end{equation*}%
It follows that
$J_{0}^{-1}\leq A^{1/\omega ^{2}}\Theta ^{-1/\omega }$,
that is,%
\begin{equation*}
	\left\Vert u\right\Vert _{L^{\infty }(B_{0}^{t})}^{\lambda _{0}-\lambda
	}\leq A^{1/\omega ^{2}}\left( \frac{C(\sigma_{0})}{\chi \left( B\right) R^{p}}%
	\int_{B^{t}}u^{\sigma }\right) ^{1/\omega },
\end{equation*} and thus,
\begin{equation}\label{withconstant0}
	\left\Vert u\right\Vert _{L^{\infty }(\frac{1}{2}B\times \left[ 0,t\right]
		)}\leq \left( \frac{A^{\frac{\lambda_{0}-\lambda}{\lambda}}C(\sigma_{0})}{\iota (B)\mu (B)R^{p}}\int_{B\times \left[ 0,t\right]
	}u^{\sigma }\right) ^{1/\lambda },
\end{equation}%
which proves (\ref{mean}) with $C=A^{\frac{\lambda_{0}-\lambda}{\lambda}}C(\sigma_{0})$.
\end{proof}

\begin{remark}\label{constantsigma0}
Using the notation from the proof of Lemma \ref{specMV}, we see from (\ref{withconstant0}) that the constant $C=A^{\frac{\lambda_{0}-\lambda}{\lambda}}C(\sigma_{0})$ in (\ref{mean}) converges to $A^{\frac{\sigma_{0}}{D}}C(\sigma_{0})$ as $\sigma\to 0+$. Hence, the limit of $C$ as $\sigma\to 0+$ is a positive constant that depends only on $p, q$ and the Faber-Krahn exponent $\nu$.		
\end{remark}	
	
	The next lemma is the main result of this section. 
	\begin{lemma}
		\label{TFPSqy}Assume that $M$ is geodesically complete and let $u$ be a bounded non-negative subsolution in $M\times %
		\left[ 0,T\right] $. Let $B=B\left( x_{0},R\right) $ be a ball such that 
		\begin{equation*}
			u_{0}=u(\cdot, 0)=0 \text{\ in\ }B.
		\end{equation*}%
		Then, for all $t\in [0, T]$,
		\begin{equation}\label{optlongdist}
			\left\Vert u\right\Vert _{L^{\infty }\left( \frac{1}{2}B\times \left[ 0,t%
				\right] \right) }\leq C\left( \frac{t}{\iota (B)R^{p}}\right) ^{\frac{1}{D }},
		\end{equation}%
		where the positive constant $C$ depends on $p, q$ and the Faber-Krahn exponent $\nu$.
	\end{lemma}

\begin{remark}
In \cite{Grigoryan2024a} the estimate (\ref{optlongdist}) was proved with additional factor $$\ln^{\gamma}\left(2+\left( \frac{\iota (B)R^{p}}{t}\right) ^{\frac{1}{D }}\frac{||u_{0}||_{L^{1}(M)}}{\mu \left(B\right) }\right),$$ where $\gamma$ is a positive constant.
\end{remark}
	
	\begin{proof}
Let $Q=B\times \left[ 0,t\right]$ and $Q^{\prime }=\frac{1}{2}B\times \left[ 0,t\right]$.
Then it follows from Lemma \ref{specMV} with $$\sigma>0\quad \textnormal{and}\quad \lambda=\sigma+D$$ and the monotonicity of $t\mapsto ||u(\cdot, t)||_{L^{\infty}(M)}$ from Lemma \ref{monl1} that \begin{align}
\left\Vert u\right\Vert _{L^{\infty }\left( Q^{\prime }\right) }&\nonumber\leq \left( 
\frac{C}{\iota (B)\mu (B)r^{p}}\int_{Q}u^{\sigma}\right) ^{1/\lambda
}\\&\nonumber\leq \left( 
\frac{C}{\iota (B)\mu (B)R^{p}}t\mu(B)||u_{0}||_{L^{\infty}(M)}^{\sigma}\right) ^{1/\lambda}\\&\label{sigmato0}=\left(\frac{Ct}{\iota (B)R^{p}}||u_{0}||_{L^{\infty}(M)}^{\sigma}\right) ^{1/\lambda}.
\end{align}
By Remark \ref{constantsigma0} the constant $C$ converges to some positive constant that depends only on $p, q$ and the Faber-Krahn exponent $\nu$ when $\sigma\to 0+$. Therefore, sending $\sigma\to 0+$ in (\ref{sigmato0}) we conclude the claim.\end{proof}
	
	\section{Long time decay}\label{longtimedecsing}
	The main result of this section is Lemma \ref{Tlong}.
\subsection{Comparison in two cylinders}

\begin{lemma}
	\label{Lem2yyy}Consider two balls $B_{0}=B\left( x_{0},r_{0}\right) $ and $%
	B_{1}=B\left( x_{0},r_{1}\right) $ with $0<r_{1}<r_{0}$ where $B_{0}$ is
	precompact. Assuming $0<t_{0}<t_{1}<T$, consider two cylinders $%
	Q_{i}=B_{i}\times \lbrack t_{i},T]$, $i=0,1.$ 
	Let $v_{0}$ be non-negative bounded subsolution in $Q_{0}.$ For  $\theta_{1}>\theta_{0} >0$  set
	\begin{equation*}
		v_{i}=\left( u-\theta_{i} \right) _{+}.
	\end{equation*}%
	Let $\sigma $ and $\lambda $ be reals satisfying (\ref{la>3-m}) and (\ref{alpha}).
	Set 
	\begin{equation*}
		J_{i}=\int_{Q_{i}}v_{i}^{\sigma }d\mu dt.
	\end{equation*}%
	Then%
	\begin{equation}\label{compsingsp1}
		J_{1}\leq \frac{Cr_{0}^{p}S^{\nu}}{\left( \iota (B_{0})\mu (B_{0})\right) ^{\nu}(\theta_{1}-\theta_{0})^{\lambda \nu}\left( r_{0}-r_{1}\right) ^{p}}\left(\frac{\theta_{1}}{\theta_{1}-\theta_{0}}\right)^{|q-1|(p-1)}J_{0} ^{1+\nu},
	\end{equation}%
	where 
	\begin{equation*}
		S=\frac{\left\Vert v_{0}\right\Vert _{L^{\infty }(Q_{0})}^{D}}{t_{1}-t_{0}}+\left(\frac{\theta_{1}}{\theta_{1}-\theta_{0}}\right)^{(q-1)(p-1)_{+}}
		\frac{1}{\left( r_{0}-r_{1}\right) ^{p}},
	\end{equation*}%
	$\nu $ is the Faber-Krahn exponent, $\iota (B_{0})$ is the Faber-Krahn
	constant in $B_{0}$, and $C$ depends on $p$, $q$ and $\lambda $.
	\FRAME{ftbphF}{2.2312in}{1.7123in}{0pt}{\Qcb{Cylinders $Q_{0}$ and $Q_{1 }$}}{}{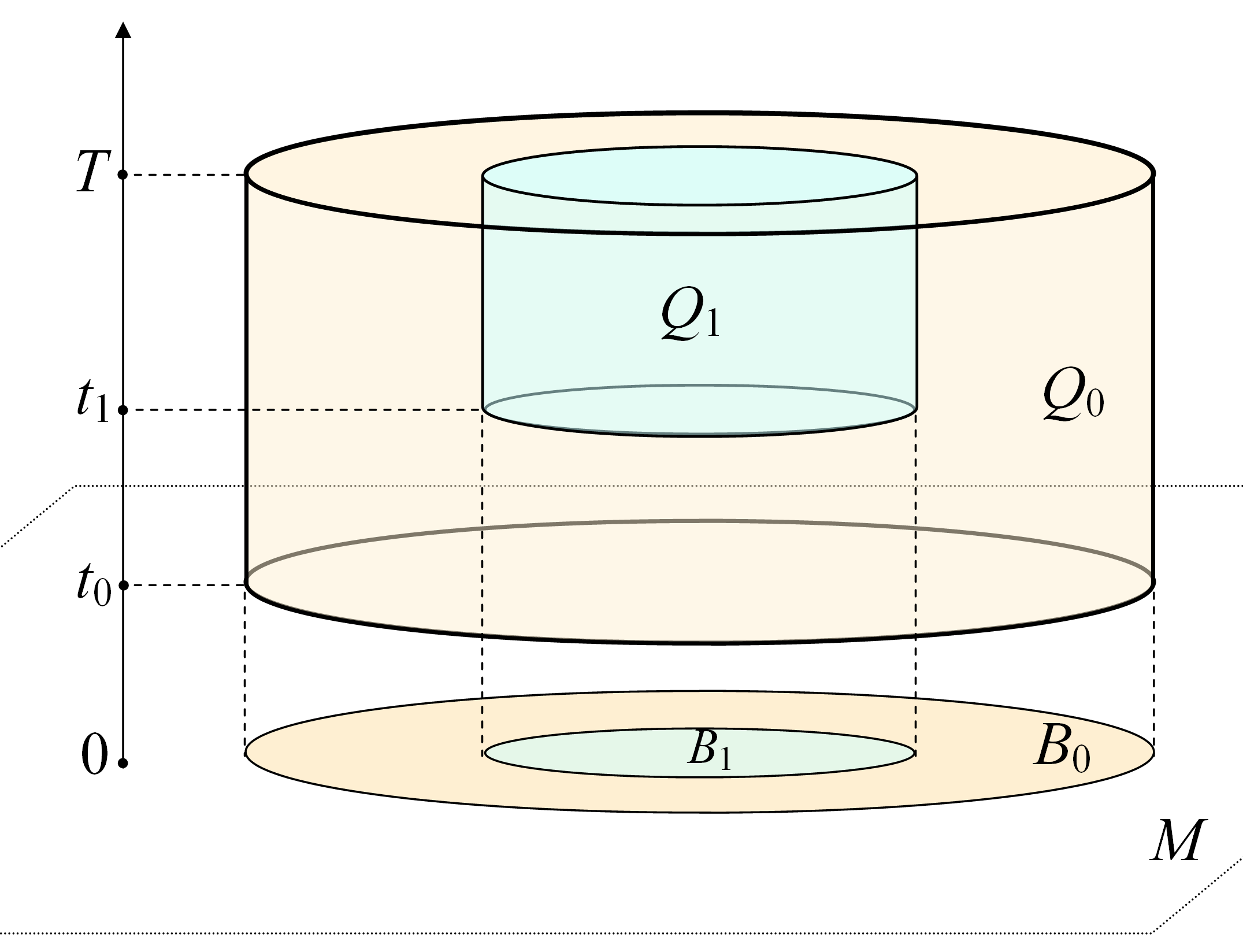}{\special{language "Scientific Word";type "GRAPHIC";maintain-aspect-ratio TRUE;display "USEDEF";valid_file "F";width
			2.2312in;height 1.7123in;depth 0pt;original-width 9.4841in;original-height
			7.2619in;cropleft "0";croptop "1";cropright "1";cropbottom "0";filename
			'pic3a.png';file-properties "NPEU";}}
\end{lemma}

\begin{proof}
	Let $\eta(x, t)=\eta \left( x\right) $ be a bump function of $B_{1}$ in $B_{1/2}:=B\left( x_{0},\frac{r_{0}+r_{1}}{2}\right)$. Recall that by (\ref{valpha}), $v_{1}^{\alpha }\eta\in L^{p}\left([t_{0},T]; W_{0}^{1, p}(B)\right)$, where $\alpha$ is defined by (\ref{alpha}), that is $\alpha=\frac{\sigma}{p}$. Hence, applying the Faber-Krahn inequality (\ref{FKp}) in ball $B_{0}$ for any $t\in [t_{0}, T]$ we get that
	\begin{equation}\label{FKapp2sing}
		\int_{B_{1}}v_{1}^{\sigma}\leq \int_{B_{0}}\left( v_{1}^{\alpha }\eta \right) ^{p}\leq r_{0}^{p}\left( \frac{\mu
			\left( D_{t}\right) }{\iota (B_{0})\mu (B_{0})}\right) ^{\nu
		} \int_{B_{0}}\left\vert \nabla \left( v_{1}^{\alpha }\eta \right) \right\vert^{p},
	\end{equation}%
	where we used that $\alpha p=\sigma$ and $\eta =1$ in $B_{1}$ and \begin{equation*}
		D_{t}=\left\{ v_{1}^{\alpha }\eta \left( \cdot ,t\right) >0\right\} =\left\{
		v_{1}>0\right\} \cap \left\{ \eta >0\right\} =\left\{ u\left( \cdot,t\right)
		>\theta_{1} \right\} \cap B_{1/2}.
	\end{equation*} Also, note that $\eta_{t}=0$ and $\left\vert \nabla \eta \right\vert \leq \frac{2}{r_{0}-r_{1}}$. From (\ref{veta1}) we therefore obtain 
	\begin{align}\nonumber
		c_{1}\int_{t_{1}}^{T}\int_{B_{0}}\left\vert \nabla \left( v_{1}^{\alpha }\eta
		\right) \right\vert ^{p}&\leq c_{2}\left(\frac{\theta_{1}}{\theta_{1}-\theta_{0}}\right)^{|q-1|(p-1)}
		\int_{t_{1}}^{T}\int_{B_{0}}v_{0}^{\sigma}|\nabla \eta|^{p}\\&\label{gradalpha3}\leq \frac{c_{3}}{\left(r_{0}-r_{1}\right) ^{p}}\left(\frac{\theta_{1}}{\theta_{1}-\theta_{0}}\right)^{|q-1|(p-1)}J_{0},
	\end{align} where $c_{3}=c_{2}2^{p}$.
	
	Let us now apply Lemma \ref{Lem1} to function $v_{0}$ in $B_{0}\times \left[t_{0}, T\right] $. Take 
	\begin{equation*}
		\eta \left( x,t\right) =\eta _{1}\left( x\right) \eta _{2}\left( t\right) ,
	\end{equation*}%
	where $\eta _{1}$ is a bump function of $B_{1/2}$ in $B_{0}$ so that 
	\begin{equation*}
		\left\vert \nabla \eta _{1}\right\vert \leq \frac{2}{r_{0}-r_{1}},
	\end{equation*}%
	and $\eta _{2}$ is a bump function of $\left[ t_{1},T\right] $ in $\left[
	t_{0},T\right] $, that is, 
	\begin{equation*}
		\eta _{2}\left( t\right) =\left\{ 
		\begin{array}{lc}
			1,\ \ t\geq t_{1} \\ 
			\frac{t-t_{0}}{t_{1}-t_{0}},\ \ t_{0}\leq t\leq t_{1}%
		\end{array}%
		\right.
	\end{equation*}%
	so that%
	\begin{equation*}
		\left\vert \partial _{t}\eta _{2}\right\vert \leq \frac{1}{t_{1}-t_{0}}.\end{equation*}
	From (\ref{veta1}) we obtain
	\begin{align*}
		\left[\int_{B_{0}}v_{0}^{\lambda}\eta ^{p}\right] _{t_{0}}^{T}&\leq 
		\int_{t_{0}}^{T}\int_{B_{0}}\left[p\eta^{p-1}\partial_{t}\eta v_{0}^{\lambda}+c_{2}\left(\frac{\theta_{1}}{\theta_{1}-\theta_{0}}\right)^{(q-1)(p-1)_{+}}\left\vert \nabla \eta \right\vert ^{p}
		v_{0}^{\sigma}\right]\\&=\int_{t_{0}}^{T}\int_{B_{0}}\left[p\eta^{p-1}\partial_{t}\eta v_{0}^{D} +c_{2}\left(\frac{\theta_{1}}{\theta_{1}-\theta_{0}}\right)^{(q-1)(p-1)_{+}}\left\vert \nabla \eta \right\vert ^{p}
		\right]v_{0}^{\sigma}.
	\end{align*}%
	Hence, for any $t\in [t_{1}, T]$, using that $\eta_{2}(t_{0})=0$ and $\eta(x, t)=1$ for $x\in B_{1/2}$ and $t\geq t_{1}$, $$\int_{B_{1/2}}v_{0}^{\lambda }\left( \cdot ,t\right)\leq c_{4} \int_{t_{0}}^{T}\int_{B_{0}}\left[\frac{||v_{0}||_{L^{\infty}}^{D}}{t_{1}-t_{0}}+\left(\frac{\theta_{1}}{\theta_{1}-\theta_{0}}\right)^{(q-1)(p-1)_{+}}\frac{1}{(r_{0}-r_{1})^{p}}
	\right]v_{0}^{\lambda}\leq c_{4}SJ_{0},$$ where $c_{4}=\max(p, c_{3})$.
	Thus, we deduce
	\begin{equation*}
		\mu \left( D_{t}\right) \leq \frac{1}{(\theta_{1}-\theta_{0}) ^{\lambda}}%
		\int_{B_{1/2}\cap \{u>\theta_{1}\}}v_{0}^{\lambda }\left( \cdot ,t\right) \leq \frac{c_{4}SJ_{0}}{%
			(\theta_{1}-\theta_{0}) ^{\lambda}}.
	\end{equation*}%
	Combining this with (\ref{FKapp2sing}) and (\ref{gradalpha3}) we obtain
	\begin{align*}
		J_{1}=\int_{t_{1}}^{T}\int_{B_{1}}v_{1}^{\sigma}&\leq r_{0}^{p}\left( \frac{c_{4}SJ_{0}}{\iota (B_{0})\mu (B_{0})(\theta_{1}-\theta_{0}) ^{\lambda}}\right) ^{\nu}
		\frac{c_{3}}{c_{1}\left(r_{0}-r_{1}\right) ^{p}}\left(\frac{\theta_{1}}{\theta_{1}-\theta_{0}}\right)^{|q-1|(p-1)}J_{0},
	\end{align*}	which implies (\ref{compsingsp1}) and finishes the proof.
\end{proof}

\subsection{Iterations and the mean value theorem}

\begin{lemma}
	\label{Tmeansing}Let the ball $B=B\left( x_{0},R\right) $ be precompact. Let $u$
	be a non-negative bounded subsolution in 
	$Q=B\times \left[ 0,T\right].$
	Let $\sigma$ and $\lambda$ be reals such that \begin{equation}\label{silapossing}\sigma>0\quad \textnormal{and}\quad \lambda=\sigma+D.\end{equation}
	Then, for the cylinder 
	\begin{equation*}
		Q^{\prime }=\frac{1}{2}B\times [ \frac{1}{2}T,T] ,
	\end{equation*}%
	we have%
	\begin{equation}
		\left\Vert u\right\Vert _{L^{\infty }\left( Q^{\prime }\right) }\leq \left( 
		\frac{CS}{\iota (B)\mu (B)}\int_{Q}u^{\sigma}\right) ^{1/\lambda},  \label{meansing}
	\end{equation}%
	where \begin{equation}\label{defS}S=\frac{\left\Vert u\right\Vert _{L^{\infty }(Q)}^{D }}{T}+%
		\frac{1}{R^{p}},\end{equation} $\iota (B)$ is the Faber-Krahn constant in $B$,
	and the constant $C$ depends on $p$, $q$, $\lambda $ and $\nu$.
	\FRAME{ftbphF}{2.2312in}{1.7123in}{0pt}{\Qcb{Cylinders $Q^{\prime}$ and $Q$}}{}{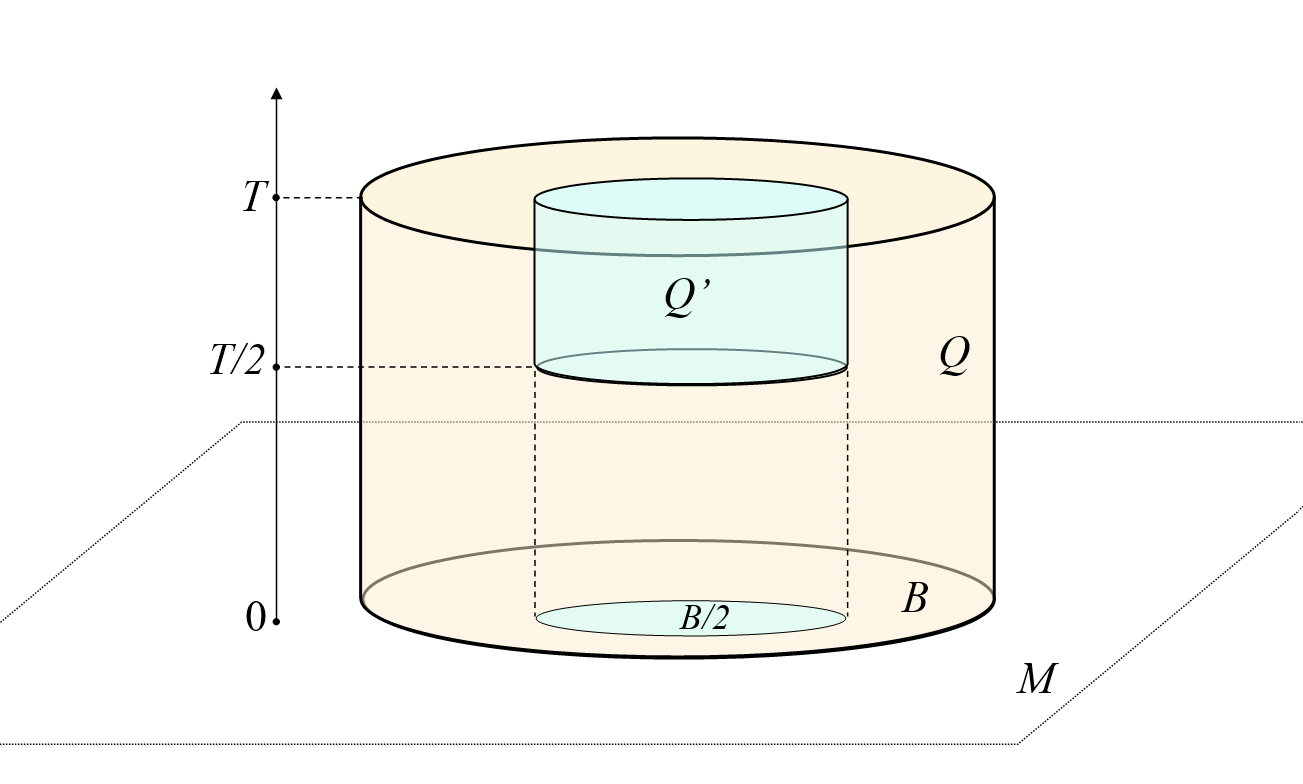}{\special{language "Scientific Word";type "GRAPHIC";maintain-aspect-ratio TRUE;display "USEDEF";valid_file "F";width
			2.2312in;height 1.7123in;depth 0pt;original-width 9.4841in;original-height
			7.2619in;cropleft "0";croptop "1";cropright "1";cropbottom "0";filename
			'pic11.png';file-properties "NPEU";}}
\end{lemma}

\begin{remark}
	In \cite{Grigoryan2024a} the same mean value inequality was proved under the condition that \begin{equation}\label{allcases} 1<p<2\quad\textnormal{and}\quad1\leq q< \frac{1}{p-1}.\end{equation}
\end{remark}

\begin{proof}
	Let us first prove (\ref{meansing}) for $\sigma$ large enough as in Lemma \ref{Lem1}. Consider sequences 
	\begin{equation*}
		r_{k}=\left( \frac{1}{2}+2^{-k-1}\right) R,\ \ \ t_{k}=\left(
		1-2^{-k}\right) \frac{T}{2}
	\end{equation*}%
	where $k=0,1,2,...$, so that
	$r_{0} =R$ and $r_{k}\searrow \frac{1}{2}R$ as $k\rightarrow
	\infty $ $t_{0} =0$ and $t_{k}\nearrow \frac{1}{2}T $ as $k\rightarrow\infty$.
	Set 
	$B_{k}=B\left( x_{0},r_{k}\right)$, $Q_{k}=B_{k}\times \left[ t_{k},T   \right]$
	so that 
	$B_{0}=B$, $Q_{0}=Q$ and $Q_{\infty}:=\lim_{k\rightarrow \infty }Q_{k}=Q^{\prime }$.
	
	Choose some $\theta >0$ to be specified later and define $\theta_{k}=\left( 1-2^{-k}\right) \theta$ and
	\begin{equation*}
		u_{k}=(u-\theta_{k})_{+}=\left( u-\left( 1-2^{-k}\right) \theta \right) _{+}.
	\end{equation*}
	
	Set 
	\begin{equation*}
		J_{k}=\int_{Q_{k}}u_{k}^{\sigma }.
	\end{equation*}%
	We obtain by Lemma \ref{Lem2yyy} that%
	\begin{equation*}
		J_{k+1}\leq \frac{Cr_{k}^{p}S_{k}^{\nu}}{\left( \iota (B_{k})\mu (B_{k})\right) ^{\nu
			}(\theta_{k+1}-\theta_{k})^{\lambda\nu}\left( r_{k}-r_{k+1}\right) ^{p}}\left(\frac{\theta_{k+1}}{\theta_{k+1}-\theta_{k}}\right)^{|q-1|(p-1)}J_{k} ^{1+\nu },
	\end{equation*}%
	where 
	\begin{equation*}
		S_{k}=\frac{\left\Vert u\right\Vert _{L^{\infty }(Q_{k})}^{D}}{t_{k+1}-t_{k}}+\left(\frac{\theta_{k+1}}{\theta_{k+1}-\theta_{k}}\right)^{(q-1)(p-1)_{+}}
		\frac{1}{\left( r_{k}-r_{k+1}\right) ^{p}}.
	\end{equation*}%
	By monotonicity of the function (\ref{cmu}), we have%
	\begin{equation*}
		\frac{r_{k}^{p}}{\left( \iota (B_{k})\mu (B_{k})\right) ^{\nu }}\leq \frac{%
			R^{p}}{\left( \iota (B)\mu (B)\right) ^{\nu }}.
	\end{equation*}%
	Since $r_{k}-r_{k+1}=2^{-k-2}R$, $t_{k+1}-t_{k}=2^{-k-2}T,$ and $\theta_{k+1}-\theta_{k}=2^{-(k+1)}\theta$
	it follows that%
	\begin{equation*}
		S_{k}\leq 2^{\left( k+2\right) (p+(q-1)(p-1)_{+})}\left( \frac{\left\Vert u\right\Vert _{L^{\infty }(Q)}^{D }}{T}+\frac{1}{R^{p}}\right)
		=2^{\left( k+2\right) (p+(q-1)(p-1)_{+})}S.
	\end{equation*}%
	Hence,%
	\begin{eqnarray*}
		J_{k+1} \leq \frac{C2^{\left( k+2\right) (\lambda\nu+p(1+\nu)+\nu(q-1)(p-1)_{+}+|q-1|(p-1))}S^{\nu}}{\left( \iota (B)\mu (B)\right) ^{\nu	}\theta^{\lambda\nu}}J_{k} ^{1+\nu}=\frac{A^{k}J_{k}^{1+\nu}}{\Theta }
	\end{eqnarray*}%
	where $$A=2^{\lambda\nu+p(1+\nu)+\nu(q-1)(p-1)_{+}+|q-1|(p-1)}\geq 1\quad\textnormal{and}\quad\Theta =c\left(\frac{ \iota (B)\mu (B)\theta ^{\lambda}}{S}\right)^{\nu}.$$
	
	Now let us apply Lemma 6.1 from \cite{grigor2024finite} with $\omega =\nu$: if 
	\begin{equation}
		\Theta \geq A^{1/\nu }J_{0}^{\nu },  \label{Thetayyy}
	\end{equation}%
	then, for all $k\geq 0,$%
	\begin{equation*}
		J_{k}\leq A^{-k/\nu}J_{0}.
	\end{equation*}%
	In terms of $\theta $ the condition (\ref{Thetayyy}) is equivalent%
	\begin{equation*}
		c\left(\frac{ \iota (B)\mu (B)\theta ^{\lambda}}{S}\right)^{\nu}\geq A^{1/\nu}J_{0}^{\nu}
	\end{equation*}%
	that is,%
	\begin{equation*}
		\theta \geq \left(\frac{CSJ_{0}}{\iota (B)\mu (B)}\right)^{1/\lambda}.
	\end{equation*}%
	Hence, we choose $\theta $ to have equality here. For this $\theta $ we
	obtain $J_{k}\rightarrow 0$ as $k\rightarrow \infty $, which implies that
	$u\leq \theta$ in $Q_{\infty }.$
	Hence,%
	\begin{equation*}
		\left\Vert u\right\Vert _{L^{\infty }\left( Q^{\prime }\right) }\leq \left( 
		\frac{CSJ_{0}}{\iota (B)\mu (B)}%
		\right) ^{1/\lambda},
	\end{equation*}%
	which proves (\ref{meansing}) for any large enough $\sigma$.
	
	By standard methods (see the proof of Lemma \ref{specMV} or \cite{Grigoryan2024}) we conclude from this that (\ref{meansing}) holds for any $\sigma >0$.
\end{proof}

	\subsection{Optimal long time decay}
	The next lemma is the main result about long time decay. 
	
	Using the mean value inequality \ref{Tmeansing} and the method from \cite{Grigoryan2024a} we can prove the following result, which improves the range of $p$ and $q$ in Lemma 5.4 from \cite{Grigoryan2024a}.
	
	\begin{lemma}
		\label{Tlong}Assume that $M$ is geodesically complete and satisfies the
		relative Faber-Krahn inequality. Assume that, for all $x\in M$ and $R\geq 1$,
		\begin{equation}
			\mu (B(x,R))\geq cR^{\alpha },  \label{Ral}
		\end{equation}%
		for some $c,\alpha >0.$ Assume also that 
		\begin{equation*}
			\beta :=p-D \alpha >0.
		\end{equation*}%
		Let $u$ be a non-negative bounded subsolution in $M\times [0, \infty)$ with initial function $u_{0}=u(\cdot, 0).$ Then, for all $t>0$, we have 
		\begin{equation}
			\left\Vert u\left( \cdot ,t\right) \right\Vert _{L^{\infty }(M)}\leq \frac{%
				C}{t^{\alpha /\beta }}\left( ||u_{0}||_{L^{1}(M)}+\left\Vert u_{0}\right\Vert _{L^{\infty }(M)}\right) ^{p/\beta },
			\label{ulong}
		\end{equation}%
		where $C$ depends on $c, \alpha ,p,q$ and on the constants in the relative Faber-Krahn inequality.
	\end{lemma}

\begin{remark}
	Note that in the result in \cite{Grigoryan2024a} we also, besides $D>0$, assumed that $$p<2\quad\textnormal{and}\quad q\geq 1$$ (see (\ref{allcases})).  
\end{remark}
	
	\section{Combined estimate}
	
	The following theorem is our main result (equivalent to Theorem \ref{singmainint} from the Introduction).
	
	\begin{theorem}\label{singmain}
		Assume that $M$ is geodesically complete and satisfies the relative
		Faber-Krahn inequality. Assume that, for all $x\in M$ and $R\geq 1,$%
		\begin{equation*}
			\mu (B(x,R))\geq cR^{\alpha },
		\end{equation*}%
		for some $c,\alpha >0.$ Assume that \emph{(\ref{signconpq})} holds and that
		\begin{equation*}
			\beta :=p-D \alpha >0.
		\end{equation*}%
		Let $u$ be a bounded non-negative subsolution in $M\times [0, \infty)$ with initial function $u_{0}=u\left( \cdot ,0\right)\in L^{1}(M)\cap L^{\infty}(M)$ and set $A=\limfunc{supp}u_{0}$. Denote 
		$\left\vert x\right\vert =d(x,A).$
		Then, for all $t>0$ and all $x\in M $, we have%
		\begin{align}
			\left\Vert u\left( \cdot ,t\right) \right\Vert _{L^{\infty }(B(x,\frac{1}{2}\left\vert x\right\vert ))}\leq \frac{C}{t^{\alpha/\beta }} \left(1+\frac{|x|}{t^{1/\beta}}\right)^{-\frac{p}{D}},  \label{uoffon}
		\end{align}%
 where the positive constant $C$ depends on $c, \alpha, p,q, ||u_{0}||_{L^{1}(M)},||u_{0}||_{L^{\infty}(M)}$ and on the constants in the relative Faber-Krahn inequality.
	\end{theorem}
	
	\begin{remark}
Theorem \ref{singmain} improves Theorem 6.1 from \cite{Grigoryan2024a} in two ways. In \cite{Grigoryan2024a} it was proved that $$\left\Vert u\left( \cdot ,t\right) \right\Vert _{L^{\infty }(B(x,\frac{1}{2}\left\vert x\right\vert ))}\leq \frac{C}{t^{\alpha/\beta }} \Phi\left(1+\frac{|x|}{t^{1/\beta}}\right)$$ with $$\Phi(s)=s^{-\frac{p}{D}}\log^{\gamma}(1+s),$$ where $\gamma$ is a positive constant. Secondly, it was additionally assumed that $$p<2\quad\textnormal{and}\quad q\geq 1.$$\end{remark}
	
	\begin{proof}
		Let us first prove that for all $t>0$ and all $x\in M\setminus A$, we have%
		\begin{eqnarray}
			\left\Vert u\left( \cdot ,t\right) \right\Vert _{L^{\infty }(B(x,\frac{1}{2}\left\vert x\right\vert ))}\leq \frac{C_{1 }}{t^{\alpha /\beta }}\wedge C_{2}\left(\frac{t}{|x|^{p}}\right)^{\frac{1}{D}}, \label{uonoff}
		\end{eqnarray}%
		where the positive constants $C_{1}, C_{2}$ depend on $c, \alpha, p,q, ||u_{0}||_{L^{1}(M)}, \left\Vert u_{0}\right\Vert _{L^{\infty }(M)}$ and on the constants
		in the relative Faber-Krahn inequality.
		
		By Lemma \ref{Tlong} we have%
		\begin{equation*}
			\left\Vert u\left( \cdot ,t\right) \right\Vert _{L^{\infty }(M)}\leq \frac{C}{t^{\alpha /\beta }}\left( ||u_{0}||_{L^{1}(M)}+\left\Vert u_{0}\right\Vert _{L^{\infty }(M)}\right)
			^{p/\beta },
		\end{equation*}%
		which gives the first term in  (\ref{uonoff}).
		In order to obtain the second term in (\ref{uonoff}), we apply Lemma \ref{TFPSqy}
		in the ball $B_{x}=B(x,\left\vert x\right\vert )$ that is disjoint with $%
		\limfunc{supp}u_{0}$ and deduce
		\begin{eqnarray*}
			\left\Vert u(\cdot ,t)\right\Vert _{L^{\infty }(\frac{1}{2}B_{x})} \leq
			C\left( \frac{t}{\iota (B_{x})|x|^{p}}\right) ^{\frac{1}{D }} \leq C\left(\frac{t}{|x|^{p}}\right)^{\frac{1}{D}}.
		\end{eqnarray*}
		
		Now let us show how (\ref{uonoff}) implies (\ref{uoffon}). In the case when $\frac{|x|}{t^{1/\beta}}\leq C^{\prime}$ for some constant $C^{\prime}>1$, we have $\left(1+\frac{|x|}{t^{1/\beta}}\right)^{-p/D}\geq \textnormal{const}>0$, which yields (\ref{uoffon}). On the other hand, if $\frac{|x|}{t^{1/\beta}}\geq C^{\prime}$, we see that $$\frac{1}{t^{\alpha/\beta }}\left(1+\frac{|x|}{t^{1/\beta}}\right)^{-p/D}\simeq \frac{t^{1/D}}{|x|^{p/D}},$$ because $\frac{p}{\beta D}-\frac{\alpha}{\beta}=\frac{1}{D}$, which finishes the proof of (\ref{uoffon}) also in this case.
	\end{proof}
	
	\begin{remark}\label{remarksharp}
		Consider a model manifold with profile $S(r)=Cr^{\alpha-1}$, for some $\alpha \in (0,n]$ and all $r\geq r_{0}$, that is, $M=(0, +\infty)\times \mathbb{S}^{n-1}$ as topological spaces and $M$ is equipped with the Riemannian metric $ds^{2}$ given by $ds^{2}=dr^{2}+\psi^{2}(r)d\theta^{2},$ where $\psi(r)$ is a smooth positive function on $(0, +\infty)$ and $d\theta^{2}$ is the standard Riemannian metric on $\mathbb{S}^{n-1}$ and $S(r)=\psi^{n-1}(r)$ (cf. Section 7.1 in \cite{Grigoryan2024a}). By Proposition 4.10 in \cite{grigor2005stability} it satisfies the relative
		Faber-Krahn inequality.
		In \cite{Grigoryan2024a} we constructed a solution which satisfies the estimate
		\begin{equation*}
			 u(r, t)\simeq \frac{1}{t^{\alpha/\beta }}\left(1+\frac{r}{t^{1/\beta}}\right)^{-p/D},
		\end{equation*}%
		which shows that our estimate (\ref{uoffon}) is sharp on such manifolds.
	\end{remark}

	\bibliographystyle{abbrv}
	\bibliography{librarycacc}
	
		\emph{Universit\"{a}t Bielefeld, Fakult\"{a}t f\"{u}r Mathematik, Postfach
		100131, D-33501, Bielefeld, Germany}
	
	\texttt{grigor@math.uni-bielefeld.de}
	
	\texttt{philipp.suerig@uni-bielefeld.de}
	
		\emph{School of Mathematical Sciences, Fudan University, 200433, Shanghai, China}
		
		\texttt{jsun22@m.fudan.edu.cn}

\end{document}